\documentclass[11pt]{amsart}
\usepackage{amssymb,latexsym}
\newtheorem{theorem}{Theorem}[section]
\newtheorem{prop}[theorem]{Proposition}
\newtheorem{lemma}[theorem]{Lemma}
\newtheorem{remark}[theorem]{Remark}
\newtheorem{question}[theorem]{Question}
\newtheorem{definition}[theorem]{Definition}
\newtheorem{cor}[theorem]{Corollary}
\newtheorem{example}[theorem]{Example}

\begin{document}

\title[Cohomology decomposition of almost complex 4-manifolds]{Symplectic forms and
cohomology decomposition of almost complex 4-manifolds}
\author{Tedi Draghici}
\address{Department of Mathematics\\Florida International University\\Miami, FL 33199}
\email{draghici@fiu.edu}
\author{Tian-Jun Li }
\address{School  of Mathematics\\  University of Minnesota\\ Minneapolis, MN 55455}
\email{tjli@math.umn.edu}
\author{Weiyi Zhang}
\address{School  of Mathematics\\  University of Minnesota\\ Minneapolis, MN 55455}
\email{zhang393@math.umn.edu}

\begin{abstract} For any compact almost complex manifold $(M,J)$,
the last two authors \cite{LZ} defined two subgroups $H_J^+(M)$,
$H_J^-(M)$ of the degree 2 real de Rham cohomology group $H^2(M,
\mathbb{R})$. These are the sets of cohomology classes which can be
represented by $J$-invariant, respectively, $J$-anti-invariant real
$2-$forms. In this note, it is shown that in dimension 4 these
subgroups induce a cohomology decomposition of $H^2(M, \mathbb{R})$.
This is a specifically 4-dimensional result, as it follows from a recent
work of Fino and Tomassini \cite{FT}. Some
estimates for the dimensions of these groups are also established
when the almost complex structure is tamed by a symplectic form and
an equivalent formulation for a question of Donaldson is given.
\end{abstract}

 \maketitle

\section{Introduction}
In this paper we continue to study differential forms on an almost
complex 4--manifold $(M, J)$ following \cite{LZ}. We are
particularly interested in the subgroups $H_J^+(M)$ and $H_J^-(M)$
of the degree 2 real de Rham cohomology group $H^2(M; \mathbb{R})$.
These are the sets of cohomology classes which can be represented by
$J$-invariant, respectively, $J$-anti-invariant real $2-$forms. The
goal pursued by defining these sub-groups is simple: understand the
effects of the action of the almost complex structure on forms at
the level of cohomology and introduce the idea of (real) cohomology
type, via the almost complex structure. Certainly, the subgroups
$H_J^{\pm}(M)$ and their dimensions $h^{\pm}_J$ are diffeomorphism
invariants of the almost complex manifold $(M,J)$. We would like to
show that these invariants appear to be interesting, particularly so
in dimension 4. Here is the outline of our paper.

\vspace{0.2cm}


Our first main result, Theorem \ref{pf-dim4} in section 2, shows
that on any compact almost complex 4-manifold the subgroups
$H_J^+(M)$ and $H_J^-(M)$ will induce a direct sum decomposition of
$H^2(M, \mathbb{R})$. With the terminology introduced in \cite{LZ},
Theorem \ref{pf-dim4} says that any almost complex structure on a
compact 4-dimensional manifold is $C^{\infty}$-pure and full. See
section 2 for precise definitions. Theorem \ref{pf-dim4} turns out to
be specifically a 4-dimensional result. Indeed, Example 3.3 of
Fino and Tomassini \cite{FT} shows the existence of a compact
6-dimensional nil-manifold with an almost complex structure which is
not $C^{\infty}$-pure (the intersection of $H_J^+(M)$ and $H_J^-(M)$
is non-empty). \footnote{We learned of the preprint \cite{FT} while
putting together the final form of our paper. There are further
interesting links between \cite{FT} and our paper (see further
comments in section 2). The overlap is minimal though.} Taking products
of this example with arbitrary almost complex manifolds, one obtains
examples in all dimensions $\geq 6$ of almost complex structures which are not
$C^{\infty}$-pure.

Also in section 2, for a compact 4-manifold with an {\bf integrable}
$J$, we show that subgroups $H_J^+(M)$ and $H_J^-(M)$ relate
naturally with the (complex) Dolbeault cohomology groups. We also
show that a complex type decomposition for cohomology does not hold
for non-integrable almost complex structures (see Lemma
\ref{Hitchin-lemma} and Corollary \ref{cxfull}).

\vspace{0.2cm}



In section 3 we focus on almost complex structures $J$ which admit
compatible or tame symplectic forms and we give estimates for the
dimensions $h^{\pm}_J$ in this case. If there are $J$-compatible
symplectic forms, then the collection of  cohomology classes of all
such forms, the so-called $J-$compatible cone, $\mathcal K_J^c(M)$,
is a subcone of $H^2(M;\mathbb R)$. In fact,
$$\mathcal K_J^c(M)\subset H_J^+(M)$$ as
 a (nonempty) open convex cone. Thus it is important
to determine  $h_J^+$ as well as $H_J^+(M)$. Moreover, it is shown
in \cite{LZ} that if $J$ is also $C^{\infty}$-full (the sum of
$H_J^+(M)$ and $H_J^-(M)$ is $H^2(M;\mathbb R)$), then
$$\mathcal
K_J^t(M)=H_J^-(M)+\mathcal K_J^c(M),$$ where $\mathcal K_J^t(M)$ is
the collection of cohomology classes of $J-$tamed symplectic forms.
Thus it is also important to understand the group $H_J^-(M)$.

Our investigation of almost complex structures which are tamed by
symplectic forms is also motivated by the following question
of Donaldson (\cite{D}).

\begin{question} \label{Donaldson}
If $J$ is an almost complex structure on a compact 4--manifold $M$
which is tamed by a symplectic form $\omega$, is there a symplectic
form compatible with $J$?
\end{question}

In \cite{LZ} it was shown that the question has an affirmative
answer when $J$ is integrable. For progress on a related problem 
proposed by Donaldson, the symplectic Calabi-Yau equation, and its relation to
Question \ref{Donaldson}, the reader is referred to \cite{D}, \cite{W}, \cite{TWY}, \cite{TW}. 

We observe in Theorem \ref{tameest} that an estimate on $h_J^+$
which is immediate for compatible $J$'s can be carried
over to the case of tamed $J$'s as well. Section 3 ends with
an equivalent formulation of Donaldson's Question \ref{Donaldson}.

\vspace{0.2cm}
In a later paper \cite{DLZ2} we will further study the group $H^-_J$.

\vspace{0.2cm}

We appreciate  V. Apostolov for his very useful comments,  R. Hind,
T. Perutz for their interest, A. Fino and A. Tomassini for
sending us their paper \cite{FT}, and NSF for the partial support.
We also thank the referees for their careful reading of the manuscript
and useful remarks.

 {\bf Convention:} The  groups indexed by $(p,q)$
arise from {\it complex} differential forms. The groups indexed by
$\pm$ arise from {\it real} differential forms.

\section{Cohomology decomposition of almost complex 4--manifolds}

\subsection{The groups  $H_J^{\pm}$}


Let $M$ be a compact $2n$-dimensional manifold and suppose $J$ is an
almost complex structure on $M$. $J$ acts on the bundle of real
2-forms $\Lambda^2$ as an involution, by $\alpha(\cdot, \cdot)
\rightarrow \alpha(J\cdot, J\cdot)$,  thus we have the splitting,
\begin{equation} \label{formtype}
\Lambda^2=\Lambda_J^+\oplus \Lambda_J^-.
\end{equation}
We will denote by $\Omega^2$ the space of 2-forms on $M$
($C^{\infty}$-sections of the bundle $\Lambda^2$) , $\Omega_J^+$ the
space of $J$-invariant 2-forms, etc. For any $\alpha \in \Omega^2$,
the $J$-invariant (resp. $J$-anti-invariant) component of $\alpha$
with respect to the decomposition (\ref{formtype}) will be denoted
by $\alpha '$ (resp. $\alpha ''$).


\begin{definition} (\cite{LZ})
Let $ \mathcal Z^2$ denote the space of closed 2-forms on $M$ and let
$\mathcal Z_J^{\pm} = \mathcal Z^2 \cap \Omega_J^{\pm}$.  Define
\begin{equation}
H_J^{\pm}(M)=\{ \mathfrak{a} \in H^2(M;\mathbb R) | \exists \; \alpha\in \mathcal
Z_J^{\pm} \mbox{ such that } [\alpha] = \mathfrak{a} \}.
\end{equation}
\end{definition}


\subsection{The type decomposition of $H^2(M;\mathbb R)$}
\noindent Obviously, $$H_J^{+}(M)+H_J^{-}(M) \subseteq H^2(M;\mathbb
R) ,$$ but if $J$ is not integrable, it is not clear whether
equality holds and whether the intersection of the two subspaces is
trivial. Thus the following definitions were also introduced in
\cite{LZ}:
\begin{definition}
(i) $J$ is said to be $C^{\infty}$-pure if $H_J^+\cap H_J^-=0$;
\newline (ii) $J$ is said to be $C^{\infty}$-full if $H^2(M;\mathbb
R)=H_J^+(M)+H_J^-(M)$.
\end{definition}

\vspace{0.2cm}

\noindent {\bf Note:} The terms {\it pure} and {\it full} almost
complex structures were also defined in \cite{LZ} in terms of
currents. We will not use these in this paper, so we refer the
reader to \cite{LZ} and \cite{FT} for more on this. Note also that
the paper of Fino and Tomassini provides a number of interesting
cases when the notions of pure and full almost complex structures
are equivalent to the $C^{\infty}$ counterparts (Theorem 3.7 and
Theorem 4.1 in \cite{FT}). See also Remark \ref{pf} below.

\vspace{0.2cm}

\noindent Our first result is
\begin{theorem} \label{pf-dim4}
If $M$ is a compact 4-dimensional manifold then any almost complex
structure $J$ on $M$ is $C^{\infty}$-pure and full. Thus, there is a
direct sum cohomology decomposition
\begin{equation}\label{type}
H^2(M;\mathbb R)=H_J^{+}(M) \oplus H_J^{-}(M).
\end{equation}
\end{theorem}

\vspace{0.2cm} \noindent Before the proof, we should set some more
preliminaries and notations. The particularity of dimension 4 is
that the Hodge operator $*_g$ of a Riemannian metric $g$ on $M$ also
acts as an involution on $\Lambda^2$. Thus, we have the well-known
self-dual, anti-self-dual splitting of the bundle of 2-forms,
\begin{equation} \label{sdasd}
\Lambda^2=\Lambda_g^+\oplus \Lambda_g^-.
\end{equation}
We will denote by $\Omega_g^{\pm}$ the space of sections of
$\Lambda_g^{\pm}$ and by $\alpha^+$, $\alpha^-$ the self-dual,
anti-self-dual components of a 2-form $\alpha$. Since the
Hodge-de Rham Laplacian commutes with $*_g$, the decomposition
(\ref{sdasd}) holds for the space of harmonic 2-forms $\mathcal H_g$
as well. By Riemannian Hodge theory, we get the metric induced
cohomology decomposition
\begin{equation} \label{hsdasd}
H^2(M;\mathbb R)= \mathcal H_g=\mathcal H_g^+\oplus \mathcal H_g^- .
\end{equation}
As in Definition 2.1, one can define
$$ H_g^{\pm}(M)=\{ \mathfrak{a} \in H^2(M;\mathbb R) | \exists \; \alpha\in \mathcal
Z_g^{\pm} \mbox{ such that } [\alpha] = \mathfrak{a} \}. $$ Of
course, $\mathcal Z_g^{\pm} := \mathcal Z^2 \cap \Omega_g^{\pm} = \mathcal H_g^{\pm}$,
so clearly $H_g^{\pm}(M)= \mathcal H_g^{\pm}$, and (\ref{hsdasd})
can be written as
$$ H^2(M;\mathbb R)= H^+_g \oplus H^-_g .$$

\vspace{0.2cm}

\noindent We will need the following special feature of the Hodge
decomposition in dimension 4.
\begin{lemma} \label{hodgedim4}
If $\alpha \in \Omega^+_g$ and $\alpha = \alpha_h + d\theta + \delta
\Psi$ is its Hodge decomposition, then $(d \theta)_g^+ = (\delta
\Psi)_g^+$ and $(d \theta)_g^- = - (\delta \Psi)_g^-$. In
particular, the $2-$form $$\alpha - 2(d\theta)_g^+ = \alpha_h$$ is
harmonic and the $2-$form
$$\alpha + 2(d\theta)_g^- = \alpha_h + 2 d\theta$$ is closed.
\end{lemma}
\begin{proof} Since $*\alpha = \alpha$, by the uniqueness of the
Hodge decomposition, we have $*(d\theta) = \delta \Psi$, $*(\delta
\Psi) = d \theta$. The lemma follows.
\end{proof}

\begin{remark}
The decomposition $\alpha  = \alpha_h + 2(d\theta)_g^+$ for a self-dual form
$\alpha$ can also be seen as the Hodge decomposition for $\Omega^+_g$ associated to the elliptic
differential complex
$$ 0 \longrightarrow \Omega^0 \buildrel {d}\over\longrightarrow \Omega^1 \buildrel {d^+}\over\longrightarrow \Omega^+_g
\longrightarrow 0 .$$
\end{remark}


Suppose now that $J$ is an almost complex structure and $g$ is a
$J$-compatible Riemannian metric on the 4-manifold $M$ in the sense
that $g$ is $J-$invariant, i.e.
$g(Ju, Jv)=g(u, v)$. The pair
$(g,J)$ defines a $J-$invariant 2--form $\omega$ by
\begin{equation}
\omega(u, v)=g(Ju, v).
\end{equation}
Such a triple $(J, g, \omega)$ is called an almost Hermitian
structure. An almost Hermitian structure $(J, g, \omega)$ is called
almost K\"ahler if $\omega$ is closed.

Given $J$, we can always choose a compatible $g$. The relations
between the decompositions (\ref{formtype}) and (\ref{sdasd}) on a
4-dimensional almost Hermitian manifold are
\begin{equation}\label{type-Jinv}
\Lambda_J^+=\underline{\mathbb R}(\omega)\oplus \Lambda_g^-,
\end{equation}
\begin{equation} \label{type-sdasd}
\Lambda_g^+ = \underline{\mathbb R}(\omega) \oplus \Lambda_J^-,
\end{equation}
\begin{equation}
\Lambda_J^+\cap \Lambda_g^+=\underline{\mathbb R}(\omega), \quad
\Lambda_J^-\cap \Lambda_g^-=0.
\end{equation}

The following lemma is an immediate consequence of (\ref{type-sdasd}):
\begin{lemma} \label{Janti-sd} Let $(M^4, g, J, \omega)$
be a 4-dimensional almost Hermitian manifold.
Then $\mathcal Z^-_J \subset \mathcal{H}^+_g$ and the natural map $\mathcal Z^-_J
\rightarrow H^-_J $ is bijective. More precisely, if
${\mathcal{H}}^{+, \omega^{\perp}}_g$ denotes the subspace of
harmonic self-dual forms point-wise orthogonal to $\omega$, we have
\begin{equation} \label{H-=Z-}
 H^-_J = \mathcal Z^-_J = {\mathcal{H}}^{+, \omega^{\perp}}_g.
\end{equation}
In particular, any closed, J-anti-invariant form $\alpha$ ($\alpha \not\equiv 0$ ) is non-degenerate
on an open dense subset $M' \subseteq M$.
\end{lemma}

\begin{proof} Since $\Lambda^-_J \subset \Lambda^+_g$, a closed
$J$-anti-invariant 2-form is a self-dual harmonic form. In
particular, there exists no non-trivial exact $J$-anti-invariant
2-form. Thus, the natural map $\mathcal Z^-_J \rightarrow H^-_J $ is
bijective. The equality (identification) (\ref{H-=Z-}) is obvious.
For the last statement, note that any self-dual form is non-degenerate
on the complement of its nodal set $M' = M \setminus \alpha^{-1}(0)$.
On the other hand, any harmonic form satisfies the unique continuation property, so if
$\alpha \not\equiv 0$, its nodal set $\alpha^{-1}(0)$ has empty interior. In
fact, from \cite{bar} it is known more: $\alpha^{-1}(0)$ has Hausdorff dimension $\leq 2$.
\end{proof}

\noindent We are now ready to give the proof of Theorem
\ref{pf-dim4}

\vspace{0.2cm}

{\it Proof of Theorem \ref{pf-dim4}.} Let $g$ be a $J$-compatible
Riemannian metric and let $\omega$ be the 2-form defined by $(g,J)$.
We start by proving that $J$ is $C^{\infty}$-pure. If $ \mathfrak{a}
\in H_J^+\cap H_J^-$, let $\alpha '\in \mathcal Z_J^{+}, \; \alpha ''\in
\mathcal Z_J^{-}$, be representatives for $\mathfrak{a}$. Then
$$\mathfrak{a}\cup \mathfrak{a} = \int_M \alpha '\wedge \alpha ''=0 ,$$
but by Lemma \ref{Janti-sd}, we also have
$$\mathfrak{a}\cup \mathfrak{a} = \int_M \alpha ''\wedge \alpha '' = \int_M
|\alpha ''|_g^2 \; d \mu_g .$$ Thus $\alpha '' = 0$, so
$\mathfrak{a} = 0$.

Next we prove that $J$ is $C^{\infty}$-full. Suppose the contrary.
Then there exists a class $ \mathfrak{a} \in H^2(M;\mathbb R)$ which
is (cup product) orthogonal to $H^+_J \oplus H^-_J$. Since $H^-_g
\subset H^+_J$, we can assume $ \mathfrak{a} \in H^+_g$. Let
$\alpha$ be the harmonic, self-dual representative of $
\mathfrak{a}$ and denote $f = <\alpha, \omega>$. The function $f$ is
not identically zero, as otherwise it follows from Lemma
\ref{Janti-sd} that $ \mathfrak{a} \in H^-_J$. Now we apply  Lemma
\ref{hodgedim4} to the self-dual form $f\omega$. The closed form
$(f\omega)_h + 2(f\omega)^{exact}$ is also $J$-invariant; indeed, it
is equal to $f \omega + 2((f \omega)^{exact})_g^{-}$. (Here and
later, we shall denote $\alpha^{exact}$ the exact part from the
Hodge decomposition of a form $\alpha$.) Thus $(f\omega)_h +
2(f\omega)^{exact}$ is a representative for a class $ \mathfrak{b}
\in H^+_J$. But
\begin{equation} \nonumber \begin{split} \mathfrak{a} \cup \mathfrak{b} &= \int_M <\alpha, (f\omega)_h + 2(f\omega)^{exact}>
\; d\mu_g \\
 &= \int_M <\alpha, f \omega + 2((f\omega)^{exact})_g^{-}> \;d\mu_g = \int_M f^2 \; d\mu_g \neq 0.\end{split}
\end{equation}
This contradicts the assumption that $ \mathfrak{a}
$ is orthogonal to $H^+_J \oplus H^-_J$. $\Box$

\vspace{0.2cm}

\begin{remark} \label{pf}
(i) Combining Theorem \ref{pf-dim4} with Theorem 3.7 from \cite{FT},
it follows that any almost complex structure on a compact
4-dimensional manifold is not just pure and full for forms, but for currents as well.

\vspace{0.1cm}

\noindent (ii) Theorem \ref{pf-dim4} does not generalize to higher dimensions.
A 6-dimensional almost complex manifold which is not $C^{\infty}$-pure is given
in Example 3.3 of \cite{FT}. Higher dimensional examples can be obtained from the following simple observation:
if $(M_1, J_1), (M_2, J_2)$ are almost complex manifolds and one of them is not $C^{\infty}$-pure, then
$(M_1 \times M_2, J_1 \oplus J_2)$ is not $C^{\infty}$-pure either.
\end{remark}

By contrast, note the following result (also proved in \cite{FT}, Proposition 3.2).

\begin{prop}
If $J$ is an almost complex structure on a compact manifold $M^{2n}$
and $J$ admits a compatible symplectic structure, then $J$ is
$C^{\infty}$-pure.
\end{prop}

\begin{proof} On any almost Hermitian manifold $(M^{2n}, g, J ,
\omega)$, if $\alpha \in \Omega^-_J$, then
\begin{equation} \label{*Janti}
*_g(\alpha) =  \alpha \wedge \omega^{n-2}.
\end{equation}
Thus, if $\omega$ is symplectic and $\alpha$ is closed,
(\ref{*Janti}) implies that $*_g(\alpha)$ is also closed. Hence, for
any almost K\"ahler structure $(g,J,\omega)$, $\mathcal Z^-_J \subset
\mathcal{H}^2_g$. It is straightforward now to generalize the first
part of the proof of Theorem \ref{pf-dim4}. Let $ \mathfrak{a} \in
H_J^+\cap H_J^-$, and let $\alpha '\in \mathcal Z_J^{+}, \; \alpha ''\in
\mathcal Z_J^{-}$, be representatives for $\mathfrak{a}$. Then
$$\mathfrak{a}\cup \mathfrak{a} \cup [\omega]^{n-2} =
\int_M \alpha '\wedge \alpha '' \wedge \omega^{n-2} =0 ,$$
but by (\ref{*Janti}) we also have
$$\mathfrak{a}\cup \mathfrak{a} \cup [\omega]^{n-2}=
\int_M \alpha ''\wedge \alpha '' \wedge \omega^{n-2} = \int_M
|\alpha ''|_g^2 \; d \mu_g .$$ Thus $\alpha '' = 0$, so
$\mathfrak{a} = 0$.
\end{proof}




\subsection{The complexified $H^2$}
\subsubsection{The groups $H_J^{p,q}$}
In all of the above, we referred to decompositions of {\it real}
2-forms. We present now the relation with the more familiar
splitting of bi-graded {\it complex} 2-forms:
\begin{equation} \label{complextype}
\Lambda^2_{\mathbb{C}} = \Lambda^{2,0}_J \oplus \Lambda^{1,1}_J
\oplus \Lambda^{0,2}_J
\end{equation}
The relation between the decompositions (\ref{formtype}) and
(\ref{complextype}) is well known:
\begin{equation} \label{realcxtype}\begin{array}{ll}
\Lambda_J^+&=(\Lambda_J^{1,1})_{\mathbb R}, \cr &\cr
\Lambda_J^-&=(\Lambda_J^{0,2}\oplus \Lambda_J^{2,0})_{\mathbb R}.\cr
\end{array}
\end{equation}
Note that the bundle $\Lambda^-_J$ inherits an almost complex
structure, still denoted $J$, by
$$\beta \in \Lambda^-_J \; \rightarrow \; J\beta \in \Lambda^-_J,
\mbox{ where } J\beta(X,Y) = -\beta(JX,Y) .$$

\begin{definition} Let $H_J^{p,q}$ be the subspace of the
complexified de Rham cohomology $H^2(M;\mathbb C)$, consisting of
classes which can be represented by a complex closed form of type
$(p,q)$.
\end{definition}

\begin{lemma} The groups $H_J^{p,q}$ have the following properties:
\begin{equation}\label{conjugation}
H_J^{p,q}=\overline{H_J^{q,p}},
\end{equation}
\begin{equation}\label{complexification} \begin{array}{ll}
H^{p,p}_{J}&= (H^{p,p}_J \cap H^{2p}(M;\mathbb R))\otimes \mathbb
C,\cr &\cr (H^{p,q}_{J}+H_J^{q,p})&= ((H^{p,q}_{J}+H_J^{q,p}) \cap
H^{p+q}(M;\mathbb R))\otimes \mathbb C.\end{array} \end{equation}
\end{lemma}
\begin{proof}  Relation \eqref{conjugation} follows from the fact 
that a complex form $\Psi$ is closed if and only if
 its conjugate  $\overline{\Psi}$ is closed. The equalities in \eqref{complexification} 
 follow from \eqref{conjugation} and  the
following fact: Let $V$ be a real vector space and $W$ a complex
subspace of $V\otimes_{\mathbb R}\mathbb C$, which as a subspace is
invariant under conjugation. Then $W$ is the complexification of
$W\cap V$ (see  Remark 2.5 on p. 139 in \cite{BPV}).
\end{proof}

We now investigate the relation between the groups $H^{\pm}_J$ and
$H_J^{p,q}$. As we shall see in Lemma \ref{Hitchin-lemma}, when $J$
is not integrable, there is an important difference compared to what
(\ref{realcxtype}) would have predicted:

\begin{lemma} For a compact almost complex manifold $(M,J)$ of any
dimension,
\begin{equation}\label{real11}
H_J^+ =H_J^{1,1}\cap H^2(M;\mathbb R),
\end{equation}
and
\begin{equation}\label{complex11}
H_J^{1,1}=H_J^+ \otimes_{\mathbb R}{\mathbb C}.
\end{equation}
\end{lemma}
\begin{proof}
The relation \eqref{complex11} is a consequence of \eqref{real11}
and \eqref{complexification} with $(p,p)=(1,1)$. So we just need to
prove \eqref{real11}.

The inclusion $ H_J^+ \subseteq H_J^{1,1}\cap H^2(M;\mathbb R)$ is
clear,  so we now prove  the converse inclusion.
 An element in $H^{1,1}_{J} \cap H^2(M; \mathbb
R)$ can be represented by a complex $d$ closed (1,1) form
$\rho=\sigma+d\tau$, with $\sigma$ a $d$ closed real form. So it is
also represented by the real $d$ closed (1,1) form
$\frac{1}{2}(\rho+\bar \rho)=\sigma+d(\tau+ \bar \tau)$.

When $J$ is integrable the same argument  appears in the proof of
Theorem 2.13 in \cite{BPV}.
\end{proof}

The next lemma is a well known result (see e.g. \cite{Sal}), recast in our terminology.
It can also be seen as a consequence and as a slight
extension of Hitchin's Lemma (\cite{Hi2}).

\begin{lemma} \label{Hitchin-lemma} Let $J$ be an almost complex structure
on a compact 4-manifold.
\begin{equation}\label{complex20}
(H_J^{2,0}+H_J^{0,2})=\begin{cases} H_J^-\otimes_{\mathbb R}{\mathbb
C},& \text{if $J$ is integrable}, \\
0,& \text{if $J$ is not integrable}.
\end{cases}
\end{equation}
In particular, if $J$ is integrable, then
\begin{equation}\label{real20} H_J^- =(H_J^{2,0}+H_J^{0,2})\cap H^2(M;\mathbb
R).
\end{equation}

\end{lemma}

\begin{proof} A (complex) form $\Phi \in \Omega_J^{2,0}$ is of the form
$$\Phi = \beta + iJ\beta, \mbox{ where } \beta \in \Omega^-_J .$$
Assume $\beta \not\equiv 0$. The point of the lemma
is that $d \beta = 0$ and $d (J\beta) = 0$ occur simultaneously if
and only if $J$ is integrable. To see this, let
$Z_j = X_j - i JX_j$, $j=1,2,3$ be arbitrary (1,0) vector fields. Then
$$d \Phi (Z_1, \overline{Z_2}, \overline{Z_3}) = - \Phi( [\overline{Z_2}, \overline{Z_3}]^{1,0}, Z_1 ). $$
Assuming $d \beta = d (J\beta) = 0$, i.e. $d \Phi = 0$, the above relation implies $[\overline{Z_2}, \overline{Z_3}]^{1,0} = 0$.
This follows first on the set $M' = M \setminus \beta^{-1}(0)$, but then everywhere on $M$ by continuity, since $M'$
is dense in $M$ (see Lemma \ref{Janti-sd}). This implies the integrability of $J$.

Conversely, assume that $J$ is integrable and we want to show that
$d\beta = 0$ iff $d(J\beta) =0$. Using $d = \partial + \bar
\partial$, and $2\beta = \Phi + \bar \Phi$, we have
$$2 d \beta = (\partial + \bar \partial) (\Phi + \bar \Phi) = \bar \partial \Phi + \partial \bar \Phi \; .$$
(We used that $\partial \Phi = 0$ since it is a (3,0) form on a complex surface.)
Thus $d \beta = 0$ iff $\bar \partial \Phi = 0$. Similarly, $d(J\beta) = 0$ iff $\bar \partial (i\Phi) = 0$.
But it is obvious that $\bar \partial \Phi = 0$ iff $\bar \partial (i\Phi) = 0$.

\end{proof}

\begin{remark} There are examples of non-integrable almost complex structures for which
the real group $H_J^-$ is non-zero, although, as shown above, the complex group
$H_J^{2,0}+H_J^{0,2}$ is always zero in this case. See Example \ref{exph-} below and
the remark that follows.
\end{remark}

\vspace{0.2cm}
  \noindent By the above two lemmas, we get:

\begin{cor} \label{cxfull} Suppose $J$ is an almost complex
structure on a compact $4-$manifold. Then $J$ is always complex
$C^{\infty}$-pure in the sense $H_J^{1,1} \cap H_J^{2,0} \cap
H_J^{0,2}=\{0\}$. Moreover, $J$ is also complex $C^{\infty}$-full,
i.e.
$$H^2(M;\mathbb C)=H_J^{1,1} \oplus H_J^{2,0} \oplus H_J^{0,2},$$ if and only if $J$ is
integrable or $h_J^-=0$.
\end{cor}

\subsubsection{Dolbeault decomposition  when $J$ is integrable}
When $J$ is integrable,  there is  the Dolbeault decomposition which
has long been discovered. We briefly recall this decomposition and
relate it to the groups $H^{p,q}_J$ introduced in the previous
subsection.


 The Fr\"ohlicher spectral
sequence of the double complex
$$(\Omega^*(M)\otimes \mathbb C=\oplus \Omega^{p,q},
\partial, \bar\partial)$$
reads (see p. 41-45, p. 140-141  in \cite{BPV}):
$$E_1^{p,q}=H^{p,q}_{\bar\partial}(M)\Rightarrow
H^{p+q}(M;\mathbb C).$$ The resulting Hodge filtration on
$H^2(M;\mathbb C)$ reads:
$$H^2(M;\mathbb C)=F^0(H^2)\supset F^1(H^2)\supset F^2(H^2)\supset
0,$$ where
\begin{equation}
F^p(H^2)=\{[\alpha], \alpha\in \oplus_{p'+q'=2, p'\geq
p}\Omega^{p',q'}|d\alpha=0\}.
\end{equation}
Since $$H^{p,q}_{\bar\partial}(M)=E_1^{p,q}\to
E_{\infty}^{p,q}=\frac{F^p(H^{p+1}(M;\mathbb
C))}{F^{p+1}(H^{p+1}(M;\mathbb C))},$$ if the Fr\"ohlicher spectral
sequence degenerates at $E_1$, then
\begin{equation} H^{p,q}_{\bar\partial}(M)\cong \frac{F^p(H^{p+1}(M;\mathbb
C))}{F^{p+1}(H^{p+1}(M;\mathbb C))}.
\end{equation}
For $p+q=2$ let
\begin{equation}
'H^{p,q}(M)=F^p(H^2)\cap \overline{F^q(H^2)}.
\end{equation}
\begin{lemma}$'H^{p,q}$ consists of de Rham classes which can be represented
by a form of type $(p, q)$, i.e
\begin{equation}\label{'}'H^{p,q}=H_J^{p,q}.
\end{equation}
\end{lemma}
This should be known to experts; we record the argument here since
it is useful to elucidate the relation between $H_J^+$ and $H_{\bar
\partial}^{1,1}$.
\begin{proof} $F^2(H^2)$ consists of de Rham classes which can be
represented by a form of type $(2, 0)$. Consequently,
$\overline{F^2(H^2)}$ consists of classes of $(0, 2)$ forms.

It remains to show that $F^1(H^2)\cap \overline{F^1(H^2)}$ consists
of de Rham classes which can be represented by a closed form of type
$(1,1)$. First of all, every such de Rham class lies in $F^1(H^2)$
and $\overline{F^1(H^2)}$. On the other hand, by definition, a class
is in $ F^1(H^2)\cap \overline{F^1(H^2)}$  if and only if it is
represented by closed forms $\alpha_1=\alpha_1^{1,1}+\alpha_1^{2,0}$
and $\alpha_2=\alpha_2^{1,1}+\alpha_2^{0,2}$. Now
$\alpha_1-\alpha_2=d\beta$, and it is easy to see that
$\alpha_1-d\beta^{1,0}=\alpha_2+d\beta^{0,1}$ is a $d$ closed (1,1)
form representing the same class.
\end{proof}
 A weight 2 {\it formal} Hodge decomposition is a decomposition of the form
\begin{equation}
H^2(M;\mathbb C)=\oplus_{p+q=2}{'H^{p,q}}.
\end{equation}

\begin{theorem} (\cite{BPV}) If $(M,J)$ is a K\"ahler manifold or a complex surface, then
the Fr\"ohlicher spectral sequence  degenerates at $E_1$, and there
is a weight 2 formal Hodge decomposition. Consequently,
\begin{equation}\label{formalHodge}
\begin{array}{lllll}
H_{\bar\partial}^{2,0}&=E_{\infty}^{2,0}&\cong F^2(H^2)&&\cong
{'H^{2,0}},\cr H_{\bar\partial}^{1,1}&=E_{\infty}^{1,1}&\cong
\frac{F^1(H^2)}{F^2(H^2)}&\cong F^1(H^2)\cap
\overline{F^1(H^2)}&\cong {'H^{1,1}},\cr
H_{\bar\partial}^{0,2}&=E_{\infty}^{0,2}&\cong \frac{H^2(M;\mathbb
C)}{F^1(H^2)}&\cong \overline{F^2(H^2)}&\cong {'H^{0,2}}.\cr
\end{array}
\end{equation}
\end{theorem}

Together with \eqref{'}, \eqref{real11} and \eqref{real20}, we
conclude
\begin{prop}\label{relations}
If $J$ is integrable on a compact 4--manifold, then
 \begin{equation} H_J^{p,q}=H_{\bar\partial}^{p,q},
\end{equation}
and
\begin{equation}\label{same} H_J^{+}=H_{\bar \partial}^{1,1}\cap H^2(M;\mathbb R),
\quad H_J^{-}=(H_{\bar \partial}^{2,0}\oplus H_{\bar
\partial}^{0,2})\cap H^2(M;\mathbb R).
\end{equation}

\end{prop}

\noindent
Let us denote  the dimension of $H_J^{\pm}$ by $h^{\pm}_J$. When $J$ is integrable, it follows from Proposition
\ref{relations} that
\begin{equation} \label{Jint-hpmcx} h_J^+=h_{\bar\partial}^{1,1}, \quad
h_J^-=2h_{\bar\partial}^{2,0}.
\end{equation}
Together with the signature theorem (Theorem 2.7 in \cite{BPV}), we
get
\begin{equation} \label{Jint-hpm}
h^+_J  =\left\{ \begin{array}{ll}
b^- +1&\hbox{if $b_1$ even} \\
b^- &\hbox{if $b_1$ odd,} \end{array} \right. \quad h^-_J = \left\{
\begin{array}{ll}
b^+ -1&\hbox{if $b_1$ even} \\
b^+ &\hbox{if $b_1$ odd.}  \end{array} \right.
\end{equation}
It is a deep, but now well known fact that the cases $b_1$ even/odd
correspond to whether the complex surface $(M,J)$ admits or not a
compatible K\"ahler structure.

\vspace{0.2cm}

Notice that when $J$ is integrable the dimensions $h^{\pm}_J$ are
topological invariants. Such properties will not hold for general
almost complex structures. In fact, we conjecture that for generic
almost complex structures $h^-_J = 0$. However, there are examples
of non-integrable almost complex structures with  $h^-_J \neq 0$.
Here is one simple construction of such examples:
\begin{example} \label{exph-}
Let $(M, g, J_0, \omega_0)$ be a compact K\"ahler surface with $b^+ \geq 3$ and let $\Phi$ be
a (not identically zero) holomorphic (2,0) form on $M$ (existence of such $\Phi$ is guaranteed by
the assumption $b^+ \geq 3$). Let $\beta = Re(\Phi)$, $J_0\beta = Im(\Phi)$ be the real and imaginary
parts of $\Phi$. Both $\beta$ and $J_0\beta$ are closed, $J_0$-anti-invariant forms. Let $f \in C^{\infty}(M)$ be
an arbitrary, not-identically zero, smooth function and consider the form $\omega_{f,\beta} = \omega_0 + f \beta$.
Because $\beta$ is pointwise orthogonal to $\omega_0$, the form $\omega_{f,\beta}$ is
non-degenerate everywhere. Since $\omega_{f,\beta}$ is also $g$-self-dual,
it induces a $g$-compatible almost complex structure $J$ on $M$. $J$ is not integrable except
the case when $f = constant$ and $(M,g,J_0)$ is hyper-K\"ahler (see, for instance \cite{AGG}).
On the other hand, $J_0 \beta$ is a non-trivial closed,
$J$-anti-invariant form. The last statement is true because $J_0\beta$ is pointwise orthogonal to
both $\omega_0$ and $\beta$. Thus $H^-_J$ is non-trivial.
\end{example}

\begin{remark} In the above example, $J$ and $J_0$ are what we call {\rm metric related}
almost complex structures, as they share a common compatible metric. In \cite{DLZ2}
we compute the exact values of $h_J^{\pm}$ for all almost complex structures
$J$ which are metric related to integrable ones. Example 6.2 of \cite{FT}
exhibits a compact 4-manifold which admits no integrable complex structures, but which admits
an almost complex structure with $h_J^- = 1$.
\end{remark}

\section{Estimates for $h^{\pm}_J$ when $J$ is tamed by a symplectic form}

From Theorem \ref{pf-dim4}, on any compact 4-dimensional almost
complex manifold $(M,J)$ we have
\begin{equation} \label{h^+_+ h^-}
h^+_J + h^-_J = b_2.
\end{equation}
The decomposition (\ref{type-sdasd}) also leads to the following
immediate estimates
\begin{equation}\label{easyestimate}
h^{+}_J\geq b^-, \quad h^-_J\leq b^+.
\end{equation}

 One reason for our interest in
$H^{\pm}_J$ stems from the following fact. If $J$ admits compatible
symplectic forms, then the set of all such forms, the $J-$compatible
cone, $\mathcal K_J^c(M)$ is a (nonempty) open convex cone of
$H_J^+(M)$ \cite{LZ}. Thus it is important to determine the
dimension $h_J^+$ of $H_J^+(M)$.

In light of the question of Donaldson mentioned in the introduction,
it is also interesting to obtain information on the dimension
$h_J^+$ in the case when $J$ is just tamed by symplectic forms.

It was shown in \cite{LZ} that an integrable $J$ admits compatible
K\"ahler structures if and only if it admits tamed symplectic forms.
Thus we can state \eqref{Jint-hpm} in this context as follows:
 \begin{equation}\label{notall}
 h^+_J=\begin{cases}
b^-+1&\hbox{if $J$ is tamed and integrable,} \\
b^-&\hbox{if $J$ is non-tamed and integrable.}
\end{cases}
\end{equation}


\subsection{A general estimate}
When $J$ admits a compatible symplectic form, we have the following
easy improvement of (\ref{easyestimate}):
\begin{prop} \label{akest}
If $J$ is almost K\"ahler, then
\begin{equation}\label{akestimate}
h^{+}_J\geq b^- + 1, \quad h^{-}_J\leq b^+-1.
\end{equation}
\end{prop}

\noindent Actually, (\ref{akestimate}) can be obtained in a slightly
more general setting:
\begin{lemma} \label{akest'}
Suppose $(M, g, J, \omega)$ is a compact 4-dimensional almost
Hermitian manifold. Assume that the harmonic part $\omega_h$ of the
Hodge decomposition of $\omega$ is not identically zero. Then
(\ref{akestimate}) holds.
\end{lemma}

\begin{proof} Let $\omega = \omega_h + d\theta + \delta \Psi$ be the
Hodge decomposition of $\omega$. From Lemma \ref{hodgedim4}, $\omega
+ 2 (d\theta)^- = \omega_h + 2 d\theta$ is a closed, $J$-invariant
2-form. By assumption, it represents a non-trivial cohomology class
in $H^+_g \cap H^+_J$ and the estimates follow.
\end{proof}

\noindent Of course, if $(M, g, J, \omega)$ is almost K\"ahler,
$\omega = \omega_h$, so Proposition \ref{akest} is obvious. More
interestingly, Lemma \ref{akest'} implies that the estimates
(\ref{akestimate}) hold for tamed $J$'s as well.


\begin{theorem} \label{tameest}
Suppose $J$ is tamed by a symplectic form $\omega$. Then the
estimates (\ref{akestimate}) still hold.
\end{theorem}

\begin{proof}
Write
\begin{equation} \omega=\omega'+\omega''
\end{equation}
with $\omega'\in \Omega_J^+$ and $\omega''\in \Omega_J^-$.
Explicitly, \begin{equation}\omega'(v, w)=\frac{1}{2}\omega(v,
w)+\frac{1}{2}\omega(Jv,Jw). \end{equation} Then $\omega'$ is
compatible with $J$ and non-degenerate, thus it determines a
Riemannian metric $g$. From the pair $(\omega, J)$ we actually get a
conformal class of metrics, these for which $\Lambda_g^+=$
Span$\{\omega, \Lambda_J^-\}$. The metric we fixed is singled out by
imposing that $|\omega'|^2=2$.

We show that Lemma \ref{akest'} can be applied to the almost
Hermitian structure $(g,J,\omega')$. It is enough to show that the
harmonic part $\omega'_h$ is not identically zero. This is true
because the following cup product is non-zero:
$$ [\omega'_h] \cup [\omega] = \int_M \omega'_h \wedge \omega =
\int_M (\omega'_h + 2 d\theta) \wedge \omega = $$
$$ = \int_M (\omega' + 2 (d\theta)_g^-) \wedge (\omega' + \omega'') =
\int_M \omega' \wedge \omega' \neq 0. $$
\end{proof}

\vspace{0.2cm}

\noindent The following is an immediate consequence:

\begin{cor}\label{b+=1t}
If $b^+=1$ and $J$ is tamed, then \begin{equation}
h^{+}_J=1+b^-=b_2, \quad h^{-}_J= b^+-1=0.
\end{equation}
\end{cor}


\begin{remark} \label{generic}  Lemma \ref{akest'} can also be applied to show that if
  $b^+\geq 1$, then the estimates (\ref{akestimate}) even hold for
generic non-tamed almost complex structures $J$ (but not all in view
of \eqref{notall}).
\end{remark}

It has been shown in \cite{Sco} that non-tamed almost complex
structures exist in any path-connected component of almost complex
structures.



\vspace{0.2cm}

\vspace{0.3cm}
\subsection{A formulation of Donaldson's question}
 We end this section by giving an equivalent formulation of
Question \ref{Donaldson}. Suppose $\tilde J$ is an almost complex
structure that is tamed by a symplectic form $\omega$ on a compact
4-manifold $M$. As noted in the proof of Theorem \ref{tameest}, the
pair $(\tilde J, \omega)$ gives rise to a conformal class of
Riemannian metrics $[g]$, so that $\Lambda^+_{[g]} = {\rm Span}
\{\omega, \Lambda^-_{\tilde J} \}$. In the proof of Theorem
\ref{tameest}, we chose in this conformal class the metric that 
made $\omega '$, the $\tilde J$-invariant part of $\omega$, have
point-wise norm $\sqrt{2}$.

For the comments below, we prefer to use another natural metric in
this conformal class: we choose the metric $g$ so that $|\omega|^2_g
= 2$ point-wise on $M$. Equivalently, $g$ is chosen so that $g$ and
$\omega$ induce an almost K\"ahler structure $(g, J, \omega)$.
Certainly, $\tilde J$ is also $g$-compatible, and let $\tilde
\omega$ be the fundamental 2-form of $(g, \tilde J)$. Then
\begin{equation} \label{ombeta}
\tilde \omega = f \omega + \gamma, \; \mbox{with $\gamma \in
\Omega^-_J$, $f \in C^{\infty}(M)$ so that $2f^2 + |\gamma|^2 = 2$.
}
\end{equation}
Since $\tilde J$ is tamed by $\omega$, the function $f$ is strictly positive on $M$. Thus, we can think that
$\tilde J$ is induced by the metric $g$ and the 2-form $\omega + \frac{1}{f}\gamma$, up to conformal
rescaling by $f$.

Conversely, let $(M^4, g, J, \omega)$ be an almost K\"ahler manifold and let $\alpha \in \Omega^-_J$.
Denote $\tilde \omega_{\alpha} = \omega + \alpha$. This is a non-degenerate, $g$ self-dual form, so
(up to a conformal normalization) it induces another $g$-compatible almost complex structure which we denote
$\tilde J_{\alpha}$. It is clear that $\tilde J_{\alpha}$ is tamed by $\omega$.

Donaldson's Question \ref{Donaldson} is equivalent to
\begin{question} \label{akDonaldson}
Is it true that for any almost  K\"ahler manifold $(M^4, g, J, \omega)$ and any $\alpha \in \Omega^-_J$,
the almost complex structure $\tilde J_{\alpha}$ is compatible with a symplectic form?
\end{question}

Using this set-up and Lemma \ref{hodgedim4}, we obtain the following
partial result.
\begin{prop} \label{simplecase}
With the notations above, if the 2-form $\alpha$ satisfies the point-wise condition
\begin{equation} \label{simplealpha}
 2 + |\alpha|^2 - 4|(\alpha^{exact})_g^-|^2 > 0 ,
\end{equation}
then $\tilde J_{\alpha}$ is compatible with  a symplectic form.
\end{prop}

\begin{proof} We just apply Lemma \ref{hodgedim4} to $\tilde \omega_{\alpha} = \omega + \alpha$. The form
$$\tilde \omega_{\alpha} + 2(\tilde \omega_{\alpha}^{exact})_g^- = \omega + \alpha + 2(\alpha^{exact})_g^-$$
is closed and  $\tilde J_{\alpha}$-invariant. Condition (\ref{simplealpha}) is equivalent to this form being point-wise
positive definite.
\end{proof}

\begin{remark}
When $\alpha$ is closed (hence harmonic), condition
(\ref{simplealpha}) is trivially satisfied. In this case, $\tilde
\omega_{\alpha}$ is itself a symplectic form. Proposition
\ref{simplecase} basically says that if $\alpha$ is not too far from
being closed, then $\tilde J_{\alpha}$ is compatible with a
symplectic form. The result can be seen in relation with the
openness result of Donaldson \cite{D}.
\end{remark}

If $\alpha$ does not satisfy (\ref{simplealpha}), Lemma
\ref{hodgedim4} may still help in the search for a symplectic form
compatible with $\tilde J_{\alpha}$. Let $(M^4, g, J, \omega)$ be
the fixed almost K\"ahler structure. Note that by \eqref{type-Jinv}
 any $\tilde J_{\alpha}$-invariant form $\Omega_{\alpha}$ can be
written as
$$ \Omega_{\alpha} = f \tilde \omega_{\alpha} + \theta , \mbox{ with } f \in C^{\infty}(M) \mbox{ and } \theta \in \Omega^-_g.$$
Applying Lemma \ref{hodgedim4} to $ f \tilde \omega_{\alpha}$, we
get that $\Omega_{\alpha}$ is also closed if and only if $\tilde \theta = \theta -
2((f \tilde \omega_{\alpha})^{exact})_g^-$ is closed, hence
harmonic. Thus, a potential symplectic form
$\Omega_{\alpha}$ which is $\tilde J_{\alpha}$-compatible must be of
the type
$$ \Omega_{\alpha} = f \tilde \omega_{\alpha} + 2((f \tilde \omega_{\alpha})^{exact})_g^- + \tilde \theta ,
\mbox{ with } f \in C^{\infty}(M) \mbox{ and } \tilde \theta \in \mathcal{H}^-_g. $$
Now the question becomes how should one choose $f \in C^{\infty}(M) \mbox{ and } \tilde \theta \in \mathcal{H}^-_g$ to satisfy
$\Omega_{\alpha}^2 > 0$ everywhere on $M$.

\vspace{0.2cm}


\begin{thebibliography}{99}





\bibitem{AGG} V. Apostolov, P. Gauduchon, G. Grantcharov, \textit{Bi-Hermitian structures on complex
surfaces}, Proc. London Math. Soc. (3) 79 (1999), no. 2, 414--428;  {\it Corrigendum},
{\bf 92} (2006), 200--202.

\bibitem{bar} C. B\"ar, \textit{On nodal sets for Dirac and Laplace operators}, Comm. Math. Phys. 188 (1997), no. 3, 709--721

\bibitem{BPV} W. Barth, K. Hulek, C. Peters, A. Van de Ven, \textit{Compact complex surfaces},
 Ergebnisse der Mathematik und ihrer Grenzgebiete. 3. Folge. A Series of Modern Surveys in Mathematics [Results in Mathematics and Related Areas. 3rd Series. A Series of Modern Surveys in Mathematics], 4. Springer-Verlag, Berlin, 2004.
\bibitem{D} S. K. Donaldson, \textit{Two-forms on
four-manifolds and elliptic equations}, Inspired by S. S. Chern,
153--172, Nankai Tracts Math., 11, World Sci. Publ., Hackensack, NJ,
2006.
\bibitem{DLZ2} T. Draghici, T. J. Li, W. Zhang, \textit{On the $J$-anti-invariant cohomology of almost complex 4-manifolds}, arXiv:1104.2511.

\bibitem{FT} A. Fino, A. Tomassini, \textit{On some cohomological
properties of almost complex manifolds}, J. Geom. Anal. 20 (2010), no. 1, 107–131, arXiv: 0811.4675v2.
\bibitem{Hi2} N. Hitchin, \textit{The self-duality equations on a Riemann surface}, Proc. London Math. Soc. 55 (1987), 59--126.
\bibitem{LZ} T.-J. Li, W. Zhang, \textit{Comparing tamed and compatible symplectic cones and cohomological properties of almost complex manifolds}, Comm. Anal. Geom. 17 (2009), no. 4, 651-683, arXiv:0708.2520.
\bibitem{Sal} S. Salamon, \textit{Special structures on four-manifolds}, Conference on Differential Geometry and Topology (Italian)
(Parma, 1991), Riv. Mat. Univ. Parma (4) 17* (1991), 109--123 (1993).
\bibitem{Sco} A. Scorpan, \textit{Existence of foliations on 4-manifolds}, Algebr. Geom. Topol. 3 (2003), 1225--1256.
\bibitem{TW}  V. Tosatti, B. Weinkove,   \textit{The Calabi-Yau equation, symplectic forms and almost complex structures}, Geometry and Analysis, No. 1, 475-493, Advanced Lectures in Math. 17 (2001), International Press, arXiv:0901.1501. 
\bibitem{TWY} V. Tosatti, B. Weinkove, S.T. Yau, \textit{Taming symplectic forms and the Calabi-Yau equation},
 Proc. Lond. Math. Soc. (3) 97 (2008),  no. 2, 401--424.
\bibitem{W} B. Weinkove, \textit{The Calabi-Yau equations on almost K\"ahler
manifolds}, J. Differential Geometry, vol 76 (2007), no. 2, 317-349.
\end{thebibliography}
\end{document}